\documentclass[12pt]{amsproc}
\usepackage{amsmath,amsthm}
\usepackage{amsfonts,amssymb}
\newtheorem{thm}{Theorem}[section]

\newtheorem{lem}[thm]{Lemma}
\newtheorem{prop}[thm]{Proposition}

\theoremstyle{remark}
\newtheorem{rem}{Remark}[section]

\def\CH{{\mathcal H}}

\def\CL{{\mathcal L}}

\def\n{{\mathfrak n}}
\def\v{{\mathfrak v}}

\def\z{{\mathfrak z}}

\def\h{{\mathfrak h}}

\def\A{{\mathbb A}}

\def\C{{\mathbb C}}
\def\H{{\mathbb H}}

\def\R{{\mathbb R}}
\def\S{{\mathbb S}}

\def\Z{{\mathbb Z}}

\def\1{\text{\bf {1}}}

\def\h{{\mathfrak h}}

\begin{document}

\title[Heat kernel transform on nilmanifolds]
{ Heat kernel transform on nilmanifolds associated to H-type groups}
\author{ A. Dasgupta and S. Thangavelu}

\address{Department of Mathematics\\ Indian Institute
of Science\\Bangalore-560 012}
\email{adgupta@math.iisc.ernet.in, veluma@math.iisc.ernet.in}

\date{\today}
\thanks{ }

\begin{abstract}
We study the heat kernel transform on a nilmanifold $ \Gamma \backslash N $
associated to a H-type group. Using a reduction technique we reduce the
problem to the case of Heisenberg groups. The image of
$ L^2(\Gamma \backslash N) $ under the heat kernel transform is shown to be a
direct sum of weighted Bergman spaces.

\end{abstract}

\maketitle

\section{Introduction}
\setcounter{equation}{0}

The aim of this article to study the heat kernel transform on  nilanifolds
associated to nilpotent groups of Heisenberg type (H-type in short). Heat
kernel transform on the Heisenberg group $ \H^n $  and the nilmanifold
associated to the standard lattice
$ \Gamma = \Z^n \times \Z^n \times \frac{1}{2}\Z $ has been studied by
Kr\"otz-Thangavelu-Xu in \cite{KTX1} and \cite{KTX2} respectively. Here we
extend their results to the general nilmanifold associated to a H-type
group.

The problem on H-type groups can be easily reduced to the case of $ \H^n $
via a partial Radon transform. This technique employed by Ricci \cite{R},
Mueller \cite{M} and others have turned out to be very useful. Thus given a
nilmanifold on a H-type group the problem can be reduced to the study of a
family of Heisenberg nilmanifolds. Lattice subgroups of the Heisenberg group
leading to nilmanifolds have been completely characterised, thanks to the
work of Tolimieri \cite{T}. Upto an automorphism, coming from the sympectic
group, they are given by lattices of the form $ \Gamma(\mathbf{l}) $ where
$\mathbf{l} = (l_1,...,l_n) $ where $ l_j $ are positive numbers such that
$ l_j $ divides $ l_{j+1} $ (see next section).

Thus we reduce the problem to the study of the heat kernel
transform on $ L^2(\Gamma(\mathbf{l})\backslash \H^n) .$ Using the
same techniques employed in \cite{KTX2} we can decompose $
L^2(\Gamma(\mathbf{l})\backslash \H^n) $ into an orthogonal direct
sum of subspaces. This is achieved by means of variants of
Weil-Brezin transforms. The image of each of the subspaces under
the heat kernel transform is then characterised. As in the case
dealt with in \cite{KTX2} we see that twisted Bergman spaces and
Hermite-Bergman spaces occur in the characterisation.

This work is closely related to \cite{KTX2} and the proofs are
modelled after the  corresponding proofs there. Hence more often
we omit proofs and supply only those which require a different
approach. In the next section we deal with heat kernel transform
on Heisenberg nilmanifolds. The case of H-type nilmanifolds is
treated in section 3.

\section{Heat kernel transform on Heisenberg nilmanifolds}
\setcounter{equation}{0}

In this section we consider the heat kernel transform for the Laplacian on a
general Heisenberg nilmanifold. Using the structure of such manifolds we
reduce the problem to the case of certain special nilmanifolds. We then use
variants of Weil-Brezin transform to study the decomposition of
the underlying $ L^2 $ space and obtain its image under the heat kernel
transform.

\subsection{Heisenberg nilmanifolds}

In this section we consider Heisenberg  nilmanifolds which are defined as
 quotients of
${\mathbb{H}}^{n}$ by certain lattice subgroups. All such lattices
have been characterised and here we recall some results without
proof  from \cite{T}. Our aim is to look at the
right regular representation of ${\mathbb{H}}^n$ on
$L^{2}(\Gamma\backslash{\mathbb{H}}^{n})$ for a general lattice and to
obtain a
decomposition of $L^{2}(\Gamma\backslash{\mathbb{H}}^{n})$ into
irreducible subspaces. The main references for this subsection are:
~\cite{AB,B,KTX2,T}.

A discrete subgroup $\Gamma$ of $\H^n$ is said to be lattice
subgroup if the quotient $\Gamma\backslash {\H^n}$ is compact. The
standard example that has been studied in details in ~\cite{KTX2}
is $\Gamma=\Z^n\times\Z^n\times{\frac{1}{2}}\Z.$ The quotient is
$\Gamma\backslash{\H^n}\simeq{{\mathbb{T}}^{2n}}\times{{\mathbb{S}}^1}$
is clearly compact. It is a circle bundle over the torus
${\mathbb{T}}^{2n}.$

Let ${\mathbb{A}}_{n}$ stand for the group of automorphism of
$\H^n$ and let ${\mathbb{A}}^{0}_{n}$ be its identity component.
This group ${\mathbb{A}}^{0}_{n}$ plays an important role in
classifying all the lattices in $\H^n$. Let $Sp(2n,\R)$ stand for
the symplectic group consisting of all $2n\times 2n$ matrices
preserving the symplectic form $\omega(\xi,\eta).$ That is,
$A\in Sp(2n,\R)$ if and only if $\omega(A\xi,A\eta)=\omega(\xi,\eta).$
Every element $A\in Sp(2n,\R)$ defines an automorphism in
${\mathbb{A}}^{0}_{n}$ denoted by the symbol $A(\xi,t)=(A\xi,t).$
Note that if $\Gamma$ is a lattice then $A(\Gamma)$ is
also another lattice. Thus we have an action ${\mathbb{A}}^{0}_{n}$
on the set $L(\H^n)$ of all lattice $\H^n.$

Observe that $\Gamma\cap Z$ is a nontrivial discrete subgroup of $Z$ and
hence there is a unique positive real number $\beta(\Gamma)$ such that
$$\Gamma\cap Z=\left\{\left(0,\beta(\Gamma): m\in
\Z\right)\right\}.$$
The lattices $\pi(\Gamma)$ where
$\pi:\H^n \rightarrow \R^{2n}$ be the projection $\pi(\xi,t)=\xi,$
satisfies the condition
$$\left[\Gamma,\Gamma\right]\subset
\Gamma\cap Z=\left\{\left(0,\beta(\Gamma): m\in
\Z\right)\right\}.$$
Indeed, when $g=(\xi,t),~ h=(\eta,s)$ then
$\left[g,h\right]=\left(0,\omega(\xi,\eta)\right).$ Actually we
have $\omega\left(\pi(\Gamma),\pi(\Gamma)\right)=\beta(\Gamma)Z$
whenever $\Gamma$ is a lattice in $\H^n$. This motivates us to make
the following definition.

A lattice $D$ in $\R^n$ is said to be a Heisenberg
lattice if $\omega(D,D)=lZ$ for some $l>0.$ The collection of such
lattices will be denoted by $HL(\R^{2n}).$ Thus $\pi(\Gamma)\in HL(\R^{2n})$
for any lattice $\Gamma$ in $\H^n$ and vice versa. We
recall the following theorem without proof from
Tolimieri~\cite{T}.

\begin{thm} There is one to one correspondence between lattices in
$\H^n$ and Heisenberg lattices in $\R^{2n}.$
\end{thm}

To each Heisenberg lattice $D$ one can
associate $n$ positive real numbers $l_1,l_2,...,l_n$ with the
property that $l_j$ divides $l_{j+1}.$ Set
$${\Z^n}^{\ast}=\left\{\mathbf{l}=(l_1,l_2,...,l_n):l_{j+1}l^{-1}_{j}\in
\Z\right\}.$$
Let $e_j,~~1\leq j\leq 2n$ be the standard coordinate vectors in
$\R^{2n}.$ For $\mathbf{l}\in {\Z^n}^{\ast}$ denote
$D(\mathbf{l})=\left[e_1,e_2,...,e_n,l_1 e_{n+1},l_2e_{n+2},...,l_n e_{2n}
\right]$ be the $\Z$ module of $\R^{2n}$ spanned
by the vectors $e_1,e_2,...,e_n,l_1e_{n+1},...,l_{n}e_{2n}.$ Then
it is clear that $D(\mathbf{l})\in HL(\R^{2n})$ and
$\omega\left(D(\mathbf{l}),D(\mathbf{l})\right)=\Z.$

\begin{thm} For each $D\in HL(\R^{2n})$ there exists a unique
$\mathbf{l}\in {\Z^n}^{\ast},$ a unique $d>0$ and an $A\in Sp(2n,\R)$
such that $D=A\left(d.D(\mathbf{l})\right).$
\end{thm}

Now combining the above the theorems we can obtain the following
result which gives the structure of all lattices in $\H^n.$ Given
$\mathbf{l}\in {\Z^n}^{\ast},$ let $\Gamma(\mathbf{l})$ be the
subgroup  of $\H^n$ generated by
$$\widetilde{e_1},\widetilde{e_2},...,\widetilde{e_{2n}}$$ where
$\widetilde{e_i}=(e_i,0)$ for $i=1,2,3,...n$ and
$\widetilde{e_j}=(l_{j}e_{j+1},0)$ for $j=1,2,...,n.$ Then we have
the following result from Tolimieri~\cite{T}. We denote the
collection of all lattices on the Heisenberg group by $L(\H^n).$

\begin{thm} For each $\Gamma\in {L(\H^n)}$ there exists a unique
$\mathbf{l}\in {\Z^n}^{\ast},$ a unique $d>0$ and an
$A\in {\mathbb{A}}_n$ such that
$\Gamma=A\left(d.\Gamma(\mathbf{l})\right).$
\end{thm}

In view of the above theorem, in studying the heat kernel transform we can
restrict ourselves to lattices of the form $ \Gamma(\mathbf{l}).$

\subsection{Analysis on the nilmanifold $\Gamma(\mathbf{l})\backslash \H^n$}
\setcounter{equation}{0}

In this subsection we consider the $\Lambda=\Gamma(\mathbf{l}),$
defined in the previous subsection and consider the nilmanifold
$M=\Gamma(\mathbf{l})\backslash \H^n$ associated to it. As we remarked earlier
there is no loss of generality in doing so. The Lebesgue measure on $\H^n$
induces an invariant measure on $M.$ So we get a unitary representation $R$
of $\H^n$ on $L^{2}(M)$ defined by
$$R(g)F\left(\Lambda h\right)=F\left(\Lambda hg\right), ~~~F\in L^{2}(M),
~~g,h \in \H^n.$$
We can identify functions on $M$ with functions on
$\H^n$ that are invariant under left translations by elements of
$\Lambda.$ Now since $\Lambda=$
span$\left\{\widetilde{e_1},\widetilde{e_2},...,\widetilde{e_{2n}}\right\},$
so it is clear that
$$ \Lambda=\Z^n\times{l_1\Z}\times{l_2\Z}\times
...\times{l_n\Z}\times{\frac{l_1}{2}\Z} $$
and any element $h\in\Lambda$ will be
$h=\left(p,q_1l_1,q_2l_2,...,q_nl_n,\frac{l_1}{2}r\right)$ where
$p,q,r\in \Z.$ To make  the notation simple we denote
$\left(q_1l_1,q_2l_2,...,q_nl_n\right)$ by $ql.$ So we can write any
element $h\in \Lambda$ as $h=\left(p,ql,\frac{l_1}{2}r\right).$

Therefore, every $\Lambda$ invariant function is $\frac{l_1}{2}$
periodic in the central variable. Thus by defining
$$H_{k}\left(\Lambda\right)=\left\{F\in L^{2}(M): F(\xi,t)=e^{\frac{4\pi
ik}{l_1}t}F(\xi,0)\right\}$$
we get the orthogonal direct sum decomposition
$$L^{2}(M)=\sum_{k\in\Z}\oplus
H_{k}\left(\Lambda\right).$$
Recall that for each  $\lambda\in \R, ~\lambda\neq 0$
the Schr\"{o}dinger representation $ \pi_\lambda $ of $\H^n$ is given by,
$$\pi_{\lambda}(x,u,\xi)\phi(v)=e^{i\lambda\xi}e^{i\lambda\left(x\cdot
v+\frac{1}{2}x\cdot u\right)}\phi(v+u).$$
It is easy to check that
each $H_{k}(\Lambda)$ is $R$ invariant and hence for every $k\neq 0$
Stone-von Neumann theorem says that the restriction of $R$ to
$H_k(\Lambda)$ decomposes into a direct sum of irreducible
representations each of which is unitarily equivalent to
$\pi_{\frac{4\pi k}{l_1}}.$

Each of the $H_{k}(\Lambda)$ can be further decomposed into
orthogonal subspaces each of which will be the image of $
L^2(\R^n) $ under a unitary operator. For the standard lattice $
\Gamma(1,1,...,1) $ such a decomposition has been obtained in
~\cite{KTX2} (see also \cite{T2}). For the general case we refer
to ~\cite{B, KTX2, T}. As we closely follow \cite{T2} in obtaining
this decomposition we will be only sketchy in our proof.

We first consider Weil-Brezin transform  $V_k$ defined on the Schwartz class
${\mathcal{S}}(\R^n)$ by
$$V_{k}f(x,u,\xi)=e^{i\lambda\xi}e^{i\frac{\lambda}{2}x\cdot u}\sum_{m\in
\Z^n}e^{i\lambda\sum^{n}_{j=1}{l_jm_jx_j}}f(u+ml)$$ where
$ml=(m_1l_1,m_2l_2,...,m_nl_n).$ It is easy to check that $V_{k}f$
is $\Lambda$ invariant. Further it can be shown that the $ L^2(M) $ norm of
$ V_kf  $ is just $ \|f\|_2 $ and hence $ V_k $ can be extended to
the whole of $ L^2(\R^n) $ as an isometry onto $H_k(\Lambda)$. We would like
to decompose $H_k(\Lambda)$ further into orthogonal subspaces.

To effect this decomposition we introduce the finite group $ {\mathbb{A}}_{k}$
which is defined by
$${\mathbb{A}}_{k}={\Z}/{2k\Z}\times {\Z}/{2kp_2\Z}\times ...
\times \Z/2kp_n\Z $$
where $p_i=\frac{l_i}{l_1}$ for $i=1,2,3,...,n.$ For each $j\in
{\mathbb{A}}_{k}$ define
$$V_{k,j}f(x,u,\xi)=e^{2\pi ij\cdot x}V_{k}f(x,u,\xi).$$
Let $ H_{k,j} $ be the image of $ L^2(\R^n) $ under $ V_{k,j}.$ Then we have
the following decomposition.

\begin{prop} For each $ k$,  $H_k$ is the orthogonal
direct sum of the spaces $H_{k,j},$ $j\in{\mathbb{A}}_{k}.$
\end{prop}

The proof of this proposition depends on several results.
The orthogonality  can be proved by direct decomposition. In what follows we
only indicate how a function $ F \in H_k $ can be decomposed into a sum of
elements of $ H_{k,j}.$

Defining $G(x,u)=F(x,u,0) $ we see that the $\Lambda$ invariance of $F$
translates into the condition
$$ G\left(x+m,u+nl\right)=e^{i\frac{\lambda}{2}(m \cdot u- x\cdot nl)}G(x,u).$$
We will show that every $G(x,u)$ satisfying the above condition can be
further decomposed as
$G(x,u)=\sum_{j\in{\mathbb{A}}_{k}}G_j,$ where $G_j$ satisfies some extra
conditions.

To this end, we  define $ G_{j,m}(x,u)$ to be
$$ e^{-\pi i\sum^{n}_{j=1}\frac{m_j}{l_j}u_j}e^{-\frac{\pi i}{k}
\sum^{n}_{i=1}\frac{m_i}{p_i}j_i}
 G\left(x_1+\frac{1}{2k}m_1,...,x_n+\frac{1}{2kp_n}m_n,u\right) $$
and consider the sum
$\sum_{j\in{\mathbb{A}}_{k}}\sum_{m\in{\mathbb{A}}_{k}}G_{j,m}(x,u)$
which is given by
$$ \sum_{m\in{\mathbb{A}}_{k}}  e^{-\pi i\sum^{n}_{j=1}\frac{m_j}{l_j}u_j}
 G\left(x_1+\frac{1}{2k}m_1,...,x_n+\frac{1}{2kp_n}m_n,u \right)
\left(\sum_{j\in{\mathbb{A}}_{k}}
e^{-\frac{\pi i}{k}\sum^{n}_{i=1}\frac{m_i}{p_i}j_i} \right).$$
Since the sum
$$ \sum_{j\in{\mathbb{A}}_{k}}e^{\frac{-\pi i}{k}
\sum^{n}_{i=1}\frac{m_i}{p_i}j_i} = 0 $$
when $ ~m\neq 0 $ and equals $ (2k)\times...\times (2kp_n )$ when $ m=0 $
we have
$$ \sum_{j\in{\mathbb{A}}_{k}}\sum_{m\in{\mathbb{A}}_{k}}G_{j,m}(x,u) =
\prod^{n}_{i=1}\left(2kp_i\right)G(x,u).$$

Therefore,  by defining
$$ G_{j}(x,u) = \prod^{n}_{i=1}\left(2kp_i\right)^{-1}
\sum_{m\in{\mathbb{A}}_{k}}G_{j,m}(x,u)$$
we get the decomposition $G=\sum_{j\in{\mathbb{A}}_{k}}G_j.$

We now claim that
$ G_j $ satisfies the extra condition
$$ G_{j}\left(x_1+\frac{1}{2k}d_1,...,x_n+\frac{1}{2kp_n}d_n,u\right) =
e^{\pi i\frac{d}{l}\cdot u} e^{\frac{\pi i}{k}\frac{d}{p}\cdot{j}}G_{j}(x,u).$$
To see this, we first observe that
$$  G_{j}\left(x_1+\frac{1}{2k}d_1,...,x_n+\frac{1}{2kp_n}d_n,u\right) =
e^{\pi i\frac{d}{l}\cdot u} e^{\frac{\pi i}{k}\frac{d}{p}\cdot{j}}
\sum_{m\in{\mathbb{A}}_{k}}G_{j,m+d}(x,u).$$ From the  definition of $ G_j $
it follows, using the quasi-periodicity of $ G $, that
$$ \sum_{m\in{\mathbb{A}}_{k}}G_{j,m}(x,u) = \sum_{m\in{\mathbb{A}}_{k}}
G_{j,m+d}(x,u) $$ and hence the claim is proved.

In order to complete the proof of the proposition we need to show that each
$ F_j(x,u,\xi) = e^{i\lambda(k)\xi}G_j(x,u) $ where $ \lambda(k) =
\frac{4\pi k}{l_1} $ can be written as $ F_j = V_{k,j}f_j $ for some
$ f_j \in L^2(\R^n).$ In order to prove this we need the following two
propositions. The first proposition deals with tempered distributions $ \nu $
on $ \R^n $ that are invariant under $ \rho_k(\Lambda) $ where $ \rho_k =
\pi_{\lambda(k)}.$  If $ \nu $ is such a ditribution then
$F(x,u,\xi)=\left(\nu,\pi_{\lambda}(x,u,\xi)f\right)$ where $f$ is
a Schwartz function on $\R^n $ gives a $ \Lambda $ invariant
function on $ \H^n.$ In view of this the following proposition plays an
important role in the decomposition of $ H_k.$

\begin{prop} Every tempered distribution $\nu$ invariant under
$\rho_k(\Lambda)$ is of the form
$\nu=\sum_{j\in {\mathbb{A}}_{k}}c_j\nu_j$ with $\nu_j$ defined by,
$$\left(\nu_j,f\right)=\sum_{m\in\Z^n}{\hat{f}\left(\frac{j_1+2km_1}{l_1},
\frac{j_2+2km_2p_2}{l_2},...,\frac{j_n+2km_np_n}{l_n}\right)}$$
Here $\hat{f}$ denotes the Fourier transform of the Schwartz class
function $f.$
\end{prop}

The proof of the above proposition is very similar to that of
Proposition 3.1 in ~\cite{KTX2} except for some technicalities,
so we skip it here.

We can also show that  the matrix
coefficients $\left(\nu_j,f\right)$ can be expressed interms of the
Weil-Brezin transforms $ V_{k,j}.$

\begin{prop} For each $f\in {\mathcal{S}}(\R^n)$
$$F_j(x,u,\xi)=V_{k,j}g_{j}\left(\frac{u}{l},-lx,\xi\right)$$
where $\frac{u}{l}=\left(\frac{u_1}{l_1},...,\frac{u_n}{l_n}\right),$
$lx=\left(l_1x_1,l_2x_2,...,l_nx_n\right)$ and $f$ and $g_j$ are
related by
$$g_j\left(l_1s_1,...,l_ns_n\right)=\hat{f}\left(\frac{j_1}{l_1}+\frac{2k}{l_1}\
s_1,...,\frac{j_n}{l_n}+\frac{2k}{l_1}s_n\right).$$
\end{prop}

For a proof this proposition in the case of the standard lattice
we refer to \cite{KTX2}. The same proof can be modified to suit
the present case. Making use of the above two propositions we can
now complete the proof of proposition 2.4. Again the proof is
similar to that of the standard lattice case. We provide a proof
just for the sake of completeness.

It only  remains to show that $G_j$ can be written as,
$$ G_j(x,u)=V_{k,j}f_j(x,u,0), ~~f_j\in L^{2}(\R^n).$$
To prove this we consider,
$$ g_j(x,u)=e^{-2\pi ij\cdot x}e^{-i\frac{\lambda}{2}x\cdot
 u}G_{j}(x,u)$$
where $\lambda=\frac{4\pi k}{l_1}$ $x=(x_1,x_2,...,x_n).$ It is clear that
$g_j$ is $\left(\frac{1}{2k},\frac{1}{2kp_2},...,\frac{1}{2kp_n}\right)$
 periodic in the $x$ variables. Therefore, it admits an expansion of
 the form
$$g_{j}(x,u)=\sum_{m\in\Z^n}c_{m}(u)e^{i\lambda ml\cdot
 x}$$
where $ml=\left(m_1l_1,m_2l_2,...,m_nl_n\right)$ and the
Fourier coefficients are given by,
$$ c_{m}(u)=\int_{[0,\frac{1}{2k})}...\int_{[0,\frac{1}{2kp_n})}\
g_{j}(x,u)e^{-i\lambda ml\cdot
 x}dx.$$
Now the transformation property of $G_j(x,u)$ leads to
$$ g_{j}(x,u-ml)=e^{i\lambda ml\cdot x}g_{j}(x,u).$$
From this relation it is then obvious that $$c_{m}(u)=c_{0}(u+ml).$$ Hence
we have
\begin{eqnarray}G_{j}&=&e^{2\pi ij\cdot x}e^{i\frac{\lambda}{2}x\cdot u}
\sum_{m\in\Z^n} c_{0}(u+ml)e^{i\lambda ml\cdot x}\\
 &=&V_{k,j}f_j(x,u,0)\end{eqnarray}
where $f_j=c_0$ and $c_0\in L^{2}(\R^n).$

\subsection{Heat kernel transform on $\Gamma(\mathbf{l})\backslash \H^n$}

In this section we consider the image of $
L^2(\Gamma(\mathbf{l})\backslash \H^n) $ under the heat kernel transform.
Let $ \Delta $ stand for the standard left invariant Laplacian on $ \H^n $
and let $ k_t $ be the associated heat kernel. It is explicitly given by
$$ k_t(x,u,\xi) = c_n \int_{-\infty}^\infty \left(
\frac{\lambda}{\sinh t\lambda}\right)^n e^{-t\lambda^2} e^{i\lambda \xi}
e^{-\frac{1}{4}(\coth t\lambda)(x^2+u^2)} d\lambda.$$
The heat semigroup is defined by $ S_tf = f*k_t $ and for $ f \in L^2(\H^n)$
it can be shown that $ S_tf $ extends to $ \C^{2n+1} $ as an entire function.
This transform taking $ f $ into the entire function $ S_tf $ is called the
heat kernel transform.

The image of $ L^2(\H^n) $ under the heat kernel transform has been studied
in \cite{KTX1}. One can also restrict the heat kernel transform to $ L^2(M) $
where $ M $ is a nilmanifold and ask for a characterisation of the image. When
$ M $ is the nilmanifold associated to the standard lattice this has been done
in \cite{KTX2}. Here we take up the general case. In view of the structure of
Heisenberg lattices, it is enough to look at the case $ \Lambda =
\Gamma(\mathbf{l}).$  Let $ S_t^\Lambda $ stand for the heat kernel transform
restricted to $ L^2(\Lambda \backslash \H^n).$ It is easy to see that
$ S_t^\Lambda $ leaves each of $ H_k $ invariant and hence it is enough to
characterise $ S_t^\Lambda(H_k) $ for each $ k =0.$ We assume $ k \neq 0 $ as
the case $ k =0 $ can be handled as in \cite{KTX2}.

In order to describe  $S_t^\Lambda(H_k)$ we define certain spaces of entire
functions. Let $ \Upsilon=\Z^n\times\left(l_1\Z\times l\
_2\Z\times...\times l_n\Z\right).$ We let
${\mathcal{H}}_{k}(\R^{2n},\Upsilon)$ stand for  the space
of all functions $F(x,u)$ for which $e^{i\lambda\xi}F(x,u)\in H_k.$
These functions are characterised by the property
$$F(x+m,u+nl)=e^{i\lambda(u\cdot m-x\cdot nl)}F(x,u)$$
for $(m,nl)\in \Upsilon.$ Then $k\neq 0$ we define
${\mathcal{H}}^{t}_{k}\left({\mathbb{C}}^{2n},\Upsilon\right)$ to
be the space all functions in ${\mathcal{H}}_{k}(\R^{2n},\Upsilon)$
having entire extension to ${\mathbb{C}}^{2n}$ and satisfying
$$||F||^{2}_{k,t}=\int_{\R^{2n}}\left( \int_{Q(l)}\int_{Q}
|F(z,w)|^{2}W^{\lambda(k)}_{t}(z,w) du dx \right) dydv <\infty $$
where $ Q(l) = [0,l_1)\times ...\times[0,l_n)$, $ Q = [0,1)^n $ and
$$ W_t^\lambda(k)(z,w) = 2^ne^{i\lambda(k)(u\cdot y-v\cdot x)} e^{-\lambda(k)
\coth (t\lambda(k))(y^2+v^2)} .$$
For each $ j \in {\mathbb{A}}_{k} $ we also define the spaces
$ {\mathcal{H}}^{t}_{k,j}\left({\mathbb{C}}^{2n},\Upsilon\right)$ as the
subspaces of
${\mathcal{H}}^{t}_{k}\left({\mathbb{C}}^{2n},\Upsilon\right)$ satisfying the
extra condition
$$ G\left(z_1+\frac{1}{2k}d_1,...,z_n+\frac{1}{2kp_n}d_n,w\right) =
e^{\pi i\frac{d}{l}\cdot u}e^{\frac{\pi i}{k}\frac{d}{p}\cdot{j}}G(z,w)$$
where $\frac{d}{p}\cdot{j}=\sum^{n}_{i=1}\frac{d_i}{p_i}j_i$ and
$\frac{d}{l}\cdot{u}=\sum^{n}_{i=1}\frac{d_i}{l_i}u_i.$ Then we have

\begin{prop}${\mathcal{H}}^{t}_{k}\left({\mathbb{\
C}}^{2n},\Upsilon\right)$ is the orthogonal direct sum of
$ {\mathcal{H}}^{t}_{k,j}\left({\mathbb{C}}^{2n},\Upsilon\right)$ as $j$ varies over ${\mathbb{A}}_{k}.$
\end{prop}

For a proof of this proposition and also for the proof of the following we
refer to \cite{KTX2}. The required modifications are left to the reader.
The relation between $ {\mathcal{H}}^{t}_{k,j}\left({\mathbb{C}}^{2n},
\Upsilon\right)$ and $ S_t^\Lambda $ is as
follows.

\begin{thm} An entire function $F(z,w)$ belongs to
${\mathcal{H}}^{t}_{k,j}\left({\mathbb{C}}^{2n},\Upsilon\right)$
if and only if $F(z,w)=e^{t\lambda^2}V_{k,j}f\ast k_t(z,w,0) $
for some $f\in L^{2}(\R^n).$
\end{thm}

The proof of this theorem given in \cite{KTX2} for the standard lattice case
makes use of Hermite-Bergman spaces. The Weil-Brezin transforms can be defined
on these Hermite-Bergman spaces and they intertwine the heat kernel transform. We refer to Proposition 4.6 in \cite{KTX2}. Finally the image of $ L^2(\Lambda
\backslash \H^n) $ under $ S_t^\Lambda $ can be described as follows.

\begin{thm} The image of $L^{2}\left(\Lambda\backslash\H^n\right)$ under
$S_t^\Lambda $ is the direct sum of $ e^{i\lambda(k)\zeta}
{\mathcal{H}}^{t}_{k,j}(\C^{2n},\Upsilon),$ $ k \in\Z,
~j\in{\mathbb{A}}_{k}$; that is,
$$S_{t}^\Lambda \left(L^{2}(\Lambda\backslash\H^n)\right)=
\sum^{\infty}_{k=-\infty}\sum_{j\in{\mathbb{A}}_{k}}e^{2t\lambda(k)^2}
e^{i\lambda(k)\zeta}{\mathcal{H}}^{t}_{k,j}(\C^{2n},\Upsilon).$$
\end{thm}

\begin{rem} Now when $\Lambda=A(d.\Gamma(1))$ then $\beta(\Lambda)=d^2$
and so $\lambda=\frac{4\pi k}{d^2}.$ Then for any $\Lambda$ invariant function
$F,$  $d^{-1}A^{-1}\circ F$ is a $\Gamma(1)$ invariant function.
So then using $V_{k,j},$ the corresponding $V_{k,j,\Lambda}$ can
be defined as $V_{k,j,\Lambda}=d^{-1}A^{-1}\circ V_{k,j}$ and
hence the same type of results can be deduced.
\end{rem}

\section{ Heat kernel transform on H-type nilmanifolds}
\setcounter{equation}{0}

In this section we study the heat kernel transform on  H-type groups $ N $
and their nilmanifolds. Using a partial Radon transform we reduce the problem
on $ N $ to a problem on $ \H^n .$ The problem on H-type nilmanifolds is also
reduced to the case of Heisenberg nilmanifolds.

\subsection{ H-type groups} H-type Lie algebras and Lie groups were
introduced by Kaplan \cite{K}. We say that a Lie algebra $ \n $ is H-type if
it is the direct sum $ \v \oplus \z $ of two Euclidean spaces with a Lie
algebra structure such that $ \z $ is the center of $ \n $ and for all unit
vector $ v \in \v $ the map $ ad(v) $ is a surjective isometry of the
orthogonal complement of $ kerad(v) $ onto $ \z.$ For such an algebra
we define a map $ J:\z \rightarrow End(\v) $ by
$$ (J_\omega v,v')  = (\omega,[v,v']), ~~~ \omega \in \z, v,v' \in\v.$$ It
then follows that $ J_\omega^2 = -I $ whenever $ \omega $ is a unit vector and
hence $ J_\omega $ defines a complex structure on $ \v.$ The Hermitian inner
product corresponding to this complex structure is given by
$$ \left\langle(v,w \right\rangle)_\omega = (v,w)+i(J_\omega v,w)=(v,w)+i([v,w],\omega).$$
Let $ 2n $
and $ m $ be  the dimensions of $ \v $ and $ \z $ respectively.

A step two nilpotent Lie group  $ N $ is said to be a H-type group if its
Lie algebra $ \n $ is of H-type. Identifying the group with its Lie algebra
we write the elements of $ N $ as $ (v,z) $ and in view of the Baker-Campbell-
Hausdorff formula the group law takes the form
$$ (v,z)(v',z') = (v,z)+(v',z')+\frac{1}{2}[(v,z),(v',z')].$$
The Heisenberg group $ \H^n $ is a H-type group. To every H-type algebra we
can associate Heisenberg algebras as follows. Given a unit vector $ \omega $
in $ \z $ let $ k(\omega) $ stand for its orthogonal complement in $ \z.$
Then the quotient algebra $ \n(\omega) = \n/k(\omega) $ can be identified
with $ \v \oplus \R $ by setting
$$ [(v,t),(v',t')]_\omega = (0,(J_\omega v,v')).$$ It has been shown in
Ricci \cite{R} that this algebra is isomorphic to the Heisenberg algebra $
\h^n.$ We denote this group by $ \H^n_\omega.$

The above connection with Heisenberg algebras makes the representation theory
of H-type groups simple. The irreducible unitary representations of $ N $
comes in two groups. The one dimensional representations do not occur in the
Plancherel formula and hence we do not consider them. If $ \pi $ is any
infinite dimensional irreducible unitary representation then its restriction
to the center is a character and hence $ \pi(0,z) = e^{i\lambda(\omega,z)}Id
$ for some $ \lambda > 0 $ and $ \omega \in S^{m-1} $ where $ m $ is the
dimension of $ \z.$ The representation $ \pi $ factors through a
representation of $ \H^n_\omega.$ By making use of Stone-von Neumann theorem
we can show that all irreducible unitary representations are parametrised by
$(\lambda,\omega) .$ We denote such a representation by
$ \pi_{\lambda,\omega} .$

The Plancherel theorem for $ N $ can be deduced from that of $
\H^n $ by making use of (partial) Radon transform. As we need to
use this we briefly recall the definition and some properties.
Given an integrable function $ f $ on $ N $ and $ \omega \in
S^{m-1} $ we define a function on $ \H^n_\omega $ by
$$ f_\omega(v,s) = \int_{k(\omega)}f(v,s\omega+\eta) d\eta.$$ The collection
$ f_\omega $ completely determines $ f $. Moreover, it can be verified that
$$ (f*g)_\omega(v,s) = f_\omega *g_\omega(v,s) $$ where the first convolution
is in $ N $  and the second in $ \H^n_\omega.$ We also remark that
$ \pi_{\lambda,\omega}(f) = \pi_\lambda(f_\omega) $ where $
\pi_\lambda $ is the Schr\"{o}dinger representation of $
\H^n_\omega.$

\subsection{Heat kernel transform on H-type groups}

We fix an orthonormal basis $ X_j, j = 1,2,..,2n $ for the Lie algebra $ \v $
and define the sublaplacian
 $ \CL = -\sum_{j=1}^{2n} X_j^2 $  as in the case of the Heisenberg groups. Then
it is known that $\CL $ generates a diffusion semigroup which is given by a
kernel $ p_t .$ This kernel has been explicitly calculated by Cygan \cite{C}
and Randall \cite{Ra}. Indeed, we have
$$ p_t(v,z) = c_n \int_{\R^m} e^{-i u\cdot z} \left(\frac{|u|}{\sinh t|u|}
\right)^n e^{-\frac{1}{4}|u|(\coth t|u|)|v|^2} du.$$
This kernel is a positive Schwartz class function and good estimates can be
proved for the same.

Let $ Z_j = \frac{\partial}{\partial z_j }, j = 1,2,...,m $ and consider the
full Laplacian $ \Delta = \CL - \sum_{j=1}^m Z_j^2 .$ The heat kernel
associated this operator is given by
$$  q_t(v,z) = c_n \int_{\R^m} e^{-t|u|^2}
e^{-i u\cdot z} \left(\frac{|u|}{\sinh t|u|}
\right)^n e^{-\frac{1}{4}|u|(\coth t|u|)|v|^2} du.$$
It then follows that $ q_t(v,z) $ can be holomorphically extended to $ \C^{2n}
\times \C^m $ as an entire function. For $ f \in L^2(N) $ the function
$ f*q_t(v,z) $ which solves the  heat equation for $ \Delta $ also extends to
$ \C^{2n} \times \C^m $ as an entire function. We are interested in this
transform taking $ f $ into the holomorphically extendable function
$ f*q_t .$ We can ask for a characterisation of the image of $ L^2(N) $ under
this transform.

In \cite{KTX1} this problem was treated for the Heisenberg group and it was
shown that the image is not a weighted Bergman transform in sharp contrast
to the Euclidean case. There the authors have obtained different
characterisations. In this section we state and prove one such
characterisation for the heat kernel transform on $ N.$ For the motivation
of the following we refer to \cite{T2}.

Consider the representations $ \pi_{\lambda,\omega} $ realised on a Hilbert
space $ \CH.$ Given $ f \in L^2(N) $ the operator $ \pi_{\lambda,\omega}(f) $
is Hilbert-Schmidt. The representations $ \pi_{\lambda,\omega} $ can also be
realised on the space of Hilbert-Schmidt operators on $ \CH $ simply defining
$ \pi_{\lambda,\omega}(x,u,\xi)T $ as the action on a Hilbert-Schmidt
operator $ T.$ Note that we have slighted changed our notation and written
$ (x,u,\xi), x,u \in \R^n, \xi \in \R^m $ for the elements of $ N.$ We can
therefore, consider the operator valued function $ (x,u,\xi) \rightarrow
\pi_{\lambda,\omega}(x,u,\xi)\pi_{\lambda,\omega}(F) $ and ask if it can be
holomorphically extended to $ \C^{n} \times \C^n \times \C^m.$ We can show
that it is so precisely when $ F = f*q_t $ for some $ f \in L^2(N).$ Thus we
get the following characterisation of the image of $ L^2(N) $ under the heat
kernel transform.

\begin{thm} A function $ F $ belongs to the image of $ L^2(N)$ under the
heat kernel transform if and only if $ (x,u,\xi) \rightarrow
\pi_{\lambda,\omega}(x,u,\xi)\pi_{\lambda,\omega}(F) $ extends to
$ \C^{n} \times \C^n \times \C^m $ as an entire function so that
$$ \|\pi_{\lambda,\omega}(i(y,v,\eta)) \pi_{\lambda,\omega}(F)^*\|_{HS} $$
is square integrable over $ N \times \R^m $ with respect to the measure
$$ e^{-\frac{1}{2t}|\eta|^2} p_{2t}^\lambda(2y,2v)|\lambda|^{n+m-1} d(y,v,\eta)d\lambda d\omega .$$
If $ F = f*q_t $ the above integral is a constant multiple of
$ \|f\|_2^2.$
\end{thm}
\begin{proof} Using partial Radon transform we can quickly reduce the theorem
to the case of the Heisenberg group. Indeed, as we have already remarked,
$ \pi_{\lambda,\omega}(F) = \pi_\lambda(F_\omega).$ We also know that
$ (q_t)_\omega = k_t,$ the heat kernel on $ \H^n_\omega.$ Therefore, if
$ F = f*q_t $ then $ \pi_{\lambda,\omega}(F) = \pi_\lambda(f_\omega *k_t).$
The function $ f_\omega $ does not belong to $ L^2(\H^n_\omega) $ but the
modified Radon transform
$$ R_\omega f(x,u,s) = D_s^{\frac{m-1}{2}}f_\omega(x,u,s) $$
where $ D_s^{\frac{m-1}{2}} $ is the fractional derivative of order
$ (m-1)/2 $ does. Since $ \pi_\lambda( R_\omega f) =
|\lambda|^{\frac{m-1}{2}}\pi_\lambda(f_\omega) $ we can write
$$ \pi_{\lambda,\omega}(x,u,\xi)\pi_{\lambda,\omega}(f*q_t) =
|\lambda|^{\frac{1-m}{2}} e^{i \lambda \omega\cdot \xi}
\pi_\lambda(x,u,0) \pi_\lambda(R_\omega f *k_t).$$ We can
therefore appeal to Theorem 13.5 in \cite{T2} to conclude that  $
(x,u,\xi) \rightarrow
\pi_{\lambda,\omega}(x,u,\xi)\pi_{\lambda,\omega}(F) $ extends to
$ \C^{n} \times \C^n \times \C^m $ as an entire function  with the
stated integrability condition.

By the same result we can also conclude that  the integral mentioned in
the statement of the theorem reduces to
$$ \int_{S^{m-1}} \left( \int_{\R^{2n+1}}|R_\omega f(x,u,s)|^2 d(x,u,s)\right)
d\omega $$
which is nothing but $ \|f\|_2^2.$ This proves half of the theorem. The other
half can be proved, again using Theorem 13.5 of \cite{T2} and the inversion
formula for the modified Radon transform.
\end{proof}

\subsection{Nilmanifolds associated to H-type groups}
\setcounter{equation}{0}

We consider a nilmanifold $ M = \Gamma \backslash N $ where $ N $ is a H-type
group and $ \Gamma $ is a lattice subgroup so that $ M $ is compact. Such
subgroups $ \Gamma $ are characterised by the property that log$\Gamma$ is a
subgroup of the underlying additive group of the Lie algebra $ \n.$ Here log
stands for the inverse of the expenential map exp:$\n \rightarrow N.$ Given
the nilmanifold $ \Gamma \backslash N $ we form $ L^2(\Gamma \backslash N) $
using an $ N-$invariant measure. We can then define the right regular
representation $ U_\Gamma $ of $ N $ on $ L^2(\Gamma \backslash N) .$ We are
interested in decomposing $ L^2(\Gamma \backslash N) $ into subspaces that are
irreducible under the action of $ U_\Gamma .$

This problem has been addressed in the more general context of nilmanifolds
in Brezin \cite{B} and Auslander-Brezin \cite{AB}. For step two nilpotent Lie
groups Brezin \cite{B} has reduced the problem to the case of Heisenberg
groups. For Heisenberg nilmanifolds he has used variants of the Weil
construction (which we have already seen in Section 2) to get the
decompostion. In the general case he has produced an algorithm which enables
one to reduce the general case to the step two case. The decomposition
obtained by Brezin in the step two case is not good enough for our the purpose
of studying heat kernel transforms. As we are dealing with a special class of
step two groups we can obtain a very explicit decomposition of
$ L^2(\Gamma \backslash N).$

By making use of the following lemma, found in M\"{u}ller
\cite{M}, we reduce the H-type case to the Heisenberg case
directly. Recall that for every $ \omega \in \S^{m-1} $ the map $
J_\omega $ defines  a complex structure on $ \v = \R^{2n}.$
Therefore, we can find an orthogonal transformation $
\sigma_\omega $ such that $ J_\omega = \sigma_\omega
J\sigma_\omega^t $ where $ J $ is the $ 2n \times 2n $ matrix
defining the standard symplectic form on $ \R^{2n}.$ We then have

\begin{lem} The mapping $ \alpha_\omega: N \rightarrow \H^n $ defined by
$ \alpha_\omega(v,z) = (\sigma_\omega^tv,z\cdot \omega), (v,z) \in N $ is
an epimorphism of Lie groups. If $ k(\omega) $ is the kernel of
$ \alpha_\omega $ and $ A = \exp k(\omega) $ then $ N/A $ is isomorphic to
$ \H^n.$
\end{lem}

Given a lattice subgroup $ \Gamma $ of $ N $ and $ \omega \in S^{m-1} $
let us set $ \Gamma_\omega = \alpha_\omega(\pi(\Gamma))$ where
$ \pi: N \rightarrow N/A $ is the canonical projection. In order to show that
$ \Gamma_\omega $ is a lattice subgroup of $ \H^n $ we
make use of the following theorem from \cite{BLU}.

\begin{thm} Suppose that $\Gamma$ is a lattice (resp. uniform lattice)
of a locally compact group $G,$ $H$ a closed normal subgroup of
the group $G $ and  $\pi:G\rightarrow{G/H}$ the canonical homomorphism.
The subgroup $\pi(\Gamma)$ is a lattice (resp. uniform lattice) in
the group $G/H$ if and only if $\Gamma\cap H$ is a lattice(resp. uniform
lattice) in the group $H.$
\end{thm}

We remark that in the terminology of the above theorem $ \Gamma $ is said to
be a uniform lattice if $ G/\Gamma $ is compact. By taking $ G = N $ and $ H
= A $ as in the lemma, we see that $ \Gamma_\omega $ is a uniform lattice
provided $ \Gamma \cap A $ is a uniform lattice in the group $ A.$ To check
that this is so we make use of another theorem from \cite{BLU}.

\begin{thm} Suppose that $\Gamma$ is a discrete subgroup and $H$ a
closed subgroup of a locally compact group $G.$ Further assume that $\Gamma$
is a uniform lattice of $G .$ Then the
subgroup $\Gamma\cap{H}$ is a uniform lattice if and only if the
subgroup $H$ is $\Gamma-$closed (i.e $H\Gamma$ is
closed).
\end{thm}

Therefore, we only need to show that $ A\Gamma $ is closed. This can be done
by following the arguments presented in Section 3 of \cite{B}. Indeed, we
need to start with the representation $ \pi_{\lambda,\omega} $ in place
of $ I_N(\phi) $ in \cite{B}, identify the kernel of the linear functional
$\phi $ and proceed with the computations. We will end up with the problem of
showing $ A\Gamma $ is closed. This is precisely the content of the Lemma
in Section 3 of \cite{B}. We refer to this article for the details.

\subsection{Heat kernel transform on H-type nilmanifolds}

We are ready to look at the heat kernel transform on a H-type nilmanifold
$ M = \Gamma \backslash N.$ We make use of the map $ \alpha_\omega $ in
order to reduce the problem to the Heisenberg nilmanifold $ \Gamma_\omega
\backslash \H^n_\omega.$

First we get a decomposition of $ L^2(\Gamma \backslash N) $ into subspaces
irreducible under the action of $ U_\Gamma.$ From  general theory (see Moore
 \cite{Mo}) it is known that $ L^2(\Gamma \backslash N) $ decomposes into a
discrete direct sum of irreducible invariant subspaces for $
U_\Gamma.$ On each of these subspaces $ U_\Gamma $ will be
unitarily equivalent to a finite multiple of $
\pi_{\lambda,\omega} $ for some $ \lambda >0 $ and $ \omega \in
S^{m-1}.$ Let $ \pi_{\lambda,\omega} $ be such a representation
which occurs in $ U_\Gamma $ and let $ K(\lambda,\omega) $ be the
invariant subspace. Recalling that $ \pi_{\lambda,\omega}(v,z) =
e^{i\lambda z\cdot \omega}\pi_{\lambda,\omega}(v,0) $ we infer
that each function $ f \in K(\lambda,\omega) $ is invariant under
the right action of $ A $ which is just $ \exp k(\omega).$ Hence
we can think of $ K(\lambda,\omega) $ as a subspace of $ L^2(N/A)
$ or even as a subspace of $ L^2\left(\Gamma_\omega \backslash
\H^n\right) $ via the map $ \alpha_\omega.$

As $ K(\lambda,\omega) $ is invariant under $ U_\Gamma $ the subspace
$$ \tilde{K}(\lambda,\omega) = \{ f: f \circ \alpha_\omega \circ \pi \in
K(\lambda,\omega) \} $$
is invariant under $ \Gamma_\omega.$ Let $ \tilde{H}_{\lambda,\omega}^j $
be the orthogonal subspaces of $ \tilde{K}(\lambda,\omega) $ on each which
$ U_{\Gamma_\omega} $ is equivalent to $ \pi_\lambda.$ Let
$ H_{\lambda,\omega}^j $ be the pull back of $ \tilde{H}_{\lambda,\omega}^j .$
Then we have

\begin{thm} Given a H-type nilmanifold $ \Gamma \backslash N $ we have
$$ L^2(\Gamma \backslash N ) = \sum_{\lambda,\omega} \sum_{j}
H_{\lambda,\omega}^j $$ where the sum is taken over all $ (\lambda, \omega) $
such that $ \pi_{\lambda,\omega} $ occurs in $ U_\Gamma.$
\end{thm}

Finally, by combining Theorem 2.9 and the above decomposition we can describe
the image of $ L^2(\Gamma \backslash N) $ under $ S_t^\Gamma.$

\begin{thm}
$$ S_t^\Gamma(L^2(\Gamma \backslash N)) = \sum_{\lambda,\omega} \sum_j
S_t^\Gamma(H_{\lambda,\omega}^j).$$
\end{thm}

In order to describe $ S_t^\Gamma(H_{\lambda,\omega}^j) $ we need to describe
$ S_t^{\Gamma_\omega}(\tilde{H}_{\lambda,\omega}^j) $ which we have done
in Theorem 2.9 when $ \Gamma_\omega = \Gamma(\mathbf{l}).$ In general there
exists $ A \in \A_n $ and $ d > 0 $ such that
$ \Gamma_\omega = A(d. \Gamma(\mathbf{l})) $ and hence an explicit description
of $ S_t^{\Gamma_\omega}(\tilde{H}_{\lambda,\omega}^j) $ can, in principle, be
written down.

\begin{center}
{\bf Acknowledgments}

\end{center}
This work is supported  by J. C. Bose Fellowship from the Department of
Science and Technology (DST).

\end{document}